\documentclass[12pt]{article}

\usepackage{latexsym,amsfonts,amsmath,amsthm,amssymb, epsfig}
\usepackage{color}

\usepackage[vcentering,dvips]{geometry}
\newtheorem{thm}{Theorem}
\newtheorem{lem}[thm]{Lemma}

\newtheorem{cor}[thm]{Corollary}

\theoremstyle{remark}
\newtheorem{rem}{Remark}

\theoremstyle{definition}

\newcommand{\R}{\mathbb{R}}

\newcommand{\N}{\mathbb{N}}

\newcommand{\rd}{\,{\rm d}}
\newcommand{\bsx}{{\boldsymbol x}}
\newcommand{\bst}{{\boldsymbol t}}
\newcommand{\bsalpha}{{\boldsymbol \alpha}}
\newcommand{\bsbeta}{{\boldsymbol \beta}}
\newcommand{\bsgamma}{{\boldsymbol \gamma}}
\newcommand{\bszero}{{\boldsymbol 0}}
\newcommand{\uu}{\mathfrak{u}}

\newcommand{\cP}{\mathcal{P}}
\newcommand{\cA}{\mathcal{A}}

\title{Conditions for tractability of the weighted $L_p$-discrepancy and 
integration in non-homogeneous tensor product spaces}  

\author{Erich Novak and  Friedrich Pillichshammer}
\date{}

\begin{document}

\maketitle

\begin{abstract}
We study tractability properties of the weighted $L_p$-discrepancy. The concept of {\it weighted} discrepancy was introduced by Sloan and Wo\'{z}\-nia\-kowski in 1998 in order to prove a weighted version of the Koksma-Hlawka 
inequality for the error of quasi-Monte Carlo integration rules. The weights have the aim to model the influence of different coordinates of integrands on the error. 
A discrepancy is said to be tractable if the information complexity, i.e., the minimal 
number $N$ of points such that the discrepancy is less than the initial discrepancy 
times an error threshold $\varepsilon$, does not grow exponentially fast with the dimension. 
In this case there are various notions of tractabilities used 
in order to classify the exact rate. 

For even integer parameters $p$ there are sufficient conditions on the 
weights available in literature, which guarantee the one or 
other notion of tractability. In the present paper we prove matching sufficient conditions (upper bounds) and neccessary conditions (lower bounds) for polynomial and weak tractability for all $p \in (1, \infty)$.

The proofs of the lower bounds are based on a general result for the information complexity 
of integration with positive quadrature formulas for tensor product spaces. 
In order to demonstrate this lower bound  
we consider as a second application the integration of 
tensor products of polynomials of degree at most 2.
\end{abstract}

\centerline{\begin{minipage}[hc]{130mm}{
{\em Keywords:} Weighted discrepancy, numerical integration, curse of dimensionality, tractability, quasi-Monte Carlo\\
{\em MSC 2010:} 11K38, 65C05, 65Y20}
\end{minipage}}

\section{Introduction and main result}\label{sec:intro}

For a set $\cP$ consisting of $N$ points $\bsx_1,\bsx_2,\ldots,\bsx_N$ in the $d$-dimensional unit-cube $[0,1)^d$ the local discrepancy function $\Delta_{\cP}:[0,1]^d \rightarrow \R$ is defined as $$\Delta_{\cP}(\bst)=\frac{|\{k \in \{1,2,\ldots,N\}\ : \ \bsx_k \in [\bszero,\bst)\}|}{N}-t_1t_2\cdots t_d,$$  for $\bst=(t_1,t_2,\ldots,t_d)$ in $[0,1]^d$, where $[\bszero,\bst)=[0,t_1)\times [0,t_2)\times \ldots \times [0,t_d)$.

The classical $L_p$-star discrepancy of a point set is defined as the $L_p$-norm of the local discrepancy function. See, e.g., \cite{BC,DT97,mat}. It is well-known that this $L_p$-star discrepancy is intimately related to the worst-case error of quasi-Monte Carlo rules for numerical integration 
via so-called Koksma-Hlawka type inequalities. 
See, e.g., \cite{H61,LP14, niesiam, NW10}. In view of this important application in 1998 Sloan and Wo\'{z}niakowski \cite{SW98} introduced a new concept of discrepancy, the so-called {\it weighted} $L_p$-discrepancy with the aim to model the influence of different coordinates of integrands on the error more sensitively. 

Let $\bsgamma=(\gamma_j)_{j \ge 1}$ be a sequence of positive real numbers, so-called coordinate weights. For $d \in \N$ let $[d]:=\{1,2,\ldots,d\}$ and for $\uu \subseteq [d]$ define $\gamma_{\uu}:=\prod_{j \in \uu} \gamma_j$. 

For a parameter $p \in [1,\infty)$ the weighted $L_p$-discrepancy of a point set $\cP$ is defined as 
\begin{equation}\label{def:wLpdisc}
L_{p,\bsgamma,N}(\cP):=\left(\sum_{\emptyset \not=\uu \subseteq [d]}\gamma_{\uu}^{p/2}\int_{[0,1]^{|\uu|}} |\Delta_{\cP}((\bst_{\uu},1))|^p \rd \bst_{\uu}\right)^{1/p},
\end{equation}
where, for a vector $\bst=(t_1,t_2,\ldots,t_d)\in [0,1]^d$ and for $\uu \subseteq [d]$ we write $\bst_{\uu}$ for the projection of $\bst$ to the coordinates which belong to $\uu$, i.e., $\bst_{\uu}=(t_j)_{j \in \uu}$ and $(\bst_{\uu},1)$ is the $d$-dimensional vector $\widetilde{\bst}=(\widetilde{t}_1,\widetilde{t}_2,\ldots,\widetilde{t}_d)$, where $\widetilde{t}_j=t_j$, if $j \in \uu$ and $\widetilde{t}_j=1$, if $j \not\in \uu$.

The weighted $L_{\infty}$-discrepancy can be introduced with the usual adaptions (see \cite{SW98}), but is not considered in the present paper.

Nowadays, the weighted $L_p$-discrepancy is a well studied quantity in literature (see, for example, \cite{DP10,KPW21,LP03,NW10,SW98}), in particular for $p=2$. It is well-known that there exists a close relation to numerical integration. This issue will be discussed in Section~\ref{sec:int}. 

For $N \in \N$ the $N$-th minimal weighted $L_p$-discrepancy in dimension $d$ is defined as $${\rm disc}_{p,\bsgamma}(N,d):=\min_{\cP} L_{p,\bsgamma,N}(\cP),$$ where the minimum is extended over all $N$-element point sets $\cP$ in $[0,1)^d$. This quantity is then compared with the initial discrepancy ${\rm disc}_{p,\bsgamma}(0,d)$, which is the weighted $L_p$-discrepancy of the empty point set. It is well known and easy to check that for the initial weighted $L_p$-discrepancy for $p \in [1,\infty)$ we have 
\begin{equation}\label{initdisc}
{\rm disc}_{p,\bsgamma}(0,d)=\prod_{j=1}^d\left(1+ \frac{\gamma_j^{p/2}}{p+1}\right)^{1/p}.
\end{equation}
See, e.g., \cite{SW98} for $p=2$ and \cite{LP03} for general finite $p$.

It is reasonable to ask for the minimal number $N$ of nodes that is necessary in order to achieve that the $N$-th minimal weighted $L_p$-discrepancy is smaller than $\varepsilon$ times ${\rm disc}_{p,\bsgamma}(0,d)$ for a threshold $\varepsilon \in (0,1)$. This quantity is called the inverse of the $N$-th minimal weighted $L_p$-discrepancy, which is, for $d \in \N$ and $\varepsilon \in (0,1)$, formally defined as $$N_{p,\bsgamma}^{{\rm disc}}(\varepsilon,d):=\min\{N \in \N \ : \ {\rm disc}_{p,\bsgamma}(N,d) \le \varepsilon \ {\rm disc}_{p,\bsgamma}(0,d)\}.$$ In a more general context, see Section~\ref{sec:int}, the inverse of the weighted discrepancy is known as ``information complexity''.

The question is now how fast $N_{p,\bsgamma}^{{\rm disc}}(\varepsilon,d)$ increases, when $d \rightarrow \infty$ and $\varepsilon \rightarrow 0$. Such  questions are typically studied in the field ``Information Based Complexity''. For general information on this and related questions we refer to the trilogy \cite{NW08,NW10,NW12}. In this context, the weighted $L_p$-discrepancy is said to be intractable, if $N_{p,\bsgamma}^{{\rm disc}}(\varepsilon,d)$ has the unfavourable property to grow at least exponentially fast in $d$ and/or $\varepsilon^{-1}$ for $d \rightarrow \infty$ and $\varepsilon \rightarrow 0$. If $N_{p,\bsgamma}^{{\rm disc}}(\varepsilon,d)$ grows at least exponentially fast in $d$ for a fixed $\varepsilon >0$, more precisely, if there exists a real $C>1$ such 
that $N_{p,\bsgamma}^{{\rm disc}}(\varepsilon,d) \ge C^d$ for infinitely 
many $d \in \N$, then the weighted $L_p$-discrepancy is said to suffer from the curse of dimensionality.

If, on the other hand, $N_{p,\bsgamma}^{{\rm disc}}(\varepsilon,d)$ grows at most sub-exponentially fast in $d$ and $\varepsilon^{-1}$ for $d \rightarrow \infty$ and $\varepsilon \rightarrow 0$, then the weighted $L_p$-discrepancy is said to be tractable and there are various notions of tractability in order to classify the grow rate of $N_{p,\bsgamma}^{{\rm disc}}(\varepsilon,d)$ more accurately. We consider the three most important ones. The weighted $L_p$-discrepancy is said to be
\begin{itemize}
\item weakly tractable, if $$\lim_{d+\varepsilon^{-1}\rightarrow \infty}\frac{\log N_{p,\bsgamma}^{{\rm disc}}(\varepsilon,d)}{d+\varepsilon^{-1}}=0;$$
\item polynomially tractable, if there exist numbers $C,\sigma>0$ and $\tau \ge 0$ such that $$N_{p,\bsgamma}^{{\rm disc}}(\varepsilon,d)\le C d^{\tau} \varepsilon^{- \sigma}\quad \mbox{for all $\varepsilon \in (0,1)$ and $d \in \N$;}$$
\item strongly polynomially tractable, if there exist numbers $C,\sigma>0$ such that 
\begin{equation}\label{def:SPT}
N_{p,\bsgamma}^{{\rm disc}}(\varepsilon,d)\le C \varepsilon^{- \sigma}\quad \mbox{for all $\varepsilon \in (0,1)$ and $d \in \N$.}
\end{equation}
The infimum of all $\sigma>0$ such that a bound of the form \eqref{def:SPT} holds is called the $\varepsilon$-exponent of strong polynomial tractability.
\end{itemize}

The problem is to characterize the weight sequences $\bsgamma$ with respect to intractability or tractability and, in the latter case, with respect to the specific notions of tractability of the $\bsgamma$-weighted $L_p$-discrepancy. 

For the $L_2$-case, i.e., $p=2$, the situation is well understood. Let $\bsgamma=(\gamma_j)_{j \ge 1}$ be a sequence of coordinate weights. It is shown in \cite{SW98} that the $\bsgamma$-weighted $L_2$-discrepancy is
\begin{itemize}
 \item polynomially tractable, if and only if $\limsup_{d \rightarrow \infty}\frac{1}{\log d}\sum_{j=1}^d \gamma_j < \infty$;
 \item strongly polynomially tractable, if and only if $\sum_{j=1}^{\infty} \gamma_j < \infty$. In this case, the $\varepsilon$-exponent of strong polynomial tractability is at most 2. 
\end{itemize}
The notion of weak tractability was introduced only several years after the
publication of \cite{SW98}. However, it follows easily from the results in \cite{SW98} that the weighted $L_2$-discrepancy is weakly tractable, if and only if $\lim_{d \rightarrow \infty}\frac{1}{d}\sum_{j=1}^d \gamma_j=0$.

For $p$ different from 2 we only have partial results consisting of sufficient conditions for even integer $p$. It was shown in \cite{LP03} that for even integer $p$ the weighted $L_p$-discrepancy is
\begin{itemize}
 \item polynomially tractable, if $\limsup_{d \rightarrow \infty}\frac{1}{\log d}\sum_{j=1}^d \gamma_j^{p/2} < \infty$;
 \item strongly polynomially tractable, if $\sum_{j=1}^{\infty} \gamma_j^{p/2} < \infty$. In this case, the $\varepsilon$-exponent of strong polynomial tractability is at most 2. 
\end{itemize}
Again, it follows easily from the results in \cite{LP03} that for even $p$ the weighted $L_p$-discrepancy is weakly tractable, if $\lim_{d \rightarrow \infty}\frac{1}{d}\sum_{j=1}^d \gamma_j^{p/2}=0$.\\

So for every even integer $p$ we have sufficient conditions for the common tractability notions. In this paper we prove necessary conditions for tractability of weighted $L_p$-discrepancy. Like in \cite{LP03,SW98} we restrict ourselves to weights $\gamma_j \in (0,1]$. This simplifies the presentation at some places and is furthermore not a big restriction, because we know from \cite{NP24b} that in the ``unweighted'' case (i.e., $\gamma_j\equiv 1$) the curse of dimensionality is present. Our main result is the following:

\begin{thm}\label{thm1}
Let $p \in (1,\infty)$ and $\bsgamma=(\gamma_j)_{j \ge 1}$ be a sequence of coordinate weights in $(0,1]$. A necessary condition for any notion of tractability of the $\bsgamma$-weighted $L_p$-discrepancy is $\lim_{j \rightarrow \infty}\gamma_j=0$. In particular, a necessary condition for
\begin{itemize}
 \item strong polynomial tractablity is 
 \begin{equation}\label{cond:SPT}
 \sum_{j=1}^{\infty}\gamma_j^{p/2}<\infty;
 \end{equation}
 \item polynomial tractablity is 
 \begin{equation}\label{cond:PT}
 \limsup_{d \rightarrow \infty}\frac{1}{\log d}\sum_{j=1}^d \gamma_j^{p/2} < \infty;
 \end{equation}
 \item weak tractability is 
 \begin{equation}\label{cond:WT}
 \lim_{d \rightarrow \infty}\frac{1}{d}\sum_{j=1}^d \gamma_j^{p/2}=0.
 \end{equation}
\end{itemize}
\end{thm}

This result will follow from a lower bound on the inverse of the weighted $L_p$-discrepancy, which in turn follows from a more general result about the integration problem in the weighted anchored Sobolev space with a $q$-norm that will be introduced and discussed in the following Section~\ref{sec:int}. This result will be stated as Corollary~\ref{cor2}.  \\

Theorem~\ref{thm1} generalizes the necessary conditions from \cite{SW98} for $p=2$ to general $p \in (1, \infty)$ and complements the sufficient conditions from \cite{LP03} for even integer $p$ with matching necessary conditions;
the case $p=1$ is still open. Beyond that, we prove also sufficient conditions 
for polynomial- and weak tractability which match the corresponding 
necessary conditions for {\it every} $p \in (1,\infty)$, rather than 
only for even $p$. Matching sufficient conditions for strong polynomial 
tractability for every $p$ beyond even integers are still open. 
Combining Theorem~\ref{thm1} with these results we obtain the following characterizations. 

\begin{cor}\label{cor1}
For coordinate weights $\bsgamma=(\gamma_j)_{j \ge 1}$ in $(0,1]$ we have:
\begin{itemize}
 \item If $p$ is even, then the $\bsgamma$-weighted $L_p$-discrepancy 
 is strongly polynomially tractable, if and only if \eqref{cond:SPT} holds.
 \item For every $p \in (1,\infty)$ the $\bsgamma$-weighted $L_p$-discrepancy 
 is polynomially tractable, if and only if \eqref{cond:PT} holds. 
 Sufficiency of \eqref{cond:PT} for polynomial tractability holds even for $p=1$.
 \item For every $p \in (1,\infty)$ the $\bsgamma$-weighted $L_p$-discrepancy 
 is weakly tractable, if and only if \eqref{cond:WT} holds. Sufficiency of \eqref{cond:WT} for weak tractability holds even for $p=1$.
\end{itemize}
\end{cor}

\begin{rem}\rm
It is worth to remark that the conditions for tractability are 
much more relaxed in the case of weighted $L_{\infty}$-discrepancy, 
better known as weighted star-discrepancy. There are a lot of papers dealing with sufficient conditions for strong polynomial tractability, see \cite{Aist2,DGPW, DLP05,DNP06,HPS08,HPT19}. The currently mildest condition on the weights $\bsgamma=(\gamma_j)_{j \ge 1}$ can be found in the paper \cite{Aist2} by Aistleitner who showed that if 
\begin{equation}\label{condA}
\sum_{j=1}^\infty \exp(-c \gamma_j^{-2}) < \infty
\end{equation}
for some $c > 0$, then there is a constant $C_{\bsgamma}>0$ such that 
$${\rm disc}_{\infty,\bsgamma}(N,d) \le \frac{C_{\bsgamma}}{\sqrt{N}} \quad \mbox{for all $d,N \in \N$.}$$ 
Consequently, the weighted star-discrepancy for such weights is strongly polynomially tractable, with $\varepsilon$-exponent at most 2. A typical sequence $\bsgamma=(\gamma_j)_{j \ge 1}$ satisfying condition \eqref{condA} is $\gamma_j=\widehat{c}/\sqrt{\log j}$ for some $\widehat{c}>0$ for $j>1$ (put $\gamma_1=1$). 
For such weights we can at best achieve weak tractability for the weighted $L_p$-discrepancy 
for finite $p$.

The observed difference corresponds to the different behavior of classical $L_p$-star discrepancy for $p \in (1,\infty)$ and $p=\infty$, which suffers from the curse of dimensionality if $p \in (1,\infty)$, see \cite{NP24}, and which is polynomially tractable if $p=\infty$, see \cite{hnww}.
\end{rem}

\section{Relation to numerical integration}\label{sec:int}

It is known that the weighted $L_p$-discrepancy is related to multivariate integration (see, for example, \cite{LP03,KPW21,NW10,SW98}). We briefly summarize the most important facts needed for the present paper.

From now on let $p,q \ge 1$ be H\"older conjugates, i.e., $1/p+1/q=1$. Let $W_q^1([0,1])$ be the space of absolutely continuous functions whose first derivatives belong to the space $L_q([0,1])$. For a generic weight $\gamma>0$ equip this space with the norm 
\begin{equation}\label{norm1}
\|f\|_{1,\gamma,q}:=\left(|f(0)|^q+\frac{1}{\gamma^{q/2}} \int_0^1 |f'(t)|^q \rd t\right)^{1/q}
\end{equation}
and let $$F_{1,\gamma,q}:=\{f \in W_q^{1}([0,1])\ : \ \|f\|_{1,\gamma,q}< \infty\}.$$

For dimension $d>1$ and a positive sequence of coordinate weights $\bsgamma=(\gamma_j)_{j \ge 1}$ consider the $d$-fold tensor product space $$F_{d,\bsgamma,q}=F_{1,\gamma_1,q} \otimes F_{1,\gamma_2,q} \otimes \ldots \otimes F_{1,\gamma_d,q}.$$ This space can be described as the space of all functions from the $d$-fold tensor product space $$W_q^{(1,1,\ldots,1)}([0,1]^d)=W_q^1([0,1])\otimes W_q^1([0,1]) \otimes \ldots \otimes W_q^1([0,1]),$$ which is the Sobolev space of functions on $[0,1]^d$ that are once differentiable in each variable and whose first derivative $\partial^d f/\partial \bsx$ has finite $L_q$-norm, where $\partial \bsx=\partial x_1 \partial x_2 \ldots \partial x_d$, and which is equipped with the tensor-product norm 
\begin{equation}\label{def:normd}
\|f\|_{d,\bsgamma,q}=\left(\sum_{\uu \subseteq [d]} \frac{1}{\gamma_{\uu}^{q/2}} \int_{[0,1]^{|\uu|}} \left| \frac{\partial^{|\uu|}}{\partial \bst_{\uu}}f_{\uu}((\bst_{\uu},0)) \right|^q \rd \bst_{\uu}\right)^{1/q},
\end{equation}
where $(\bst_{\uu},0)$ is defined analogously to $(\bst_{\uu},1)$ with $1$ replaced by $0$.

Consider multivariate integration $$I_d(f):=\int_{[0,1]^d} f(\bsx) \rd \bsx \quad \mbox{for $f \in F_{d,\bsgamma,q}$}.$$
We approximate the integrals $I_d(f)$ by linear algorithms of the form 
\begin{equation}\label{def:linAlg}
A_{d,N}(f)=\sum_{k=1}^N a_k f(\bsx_k),
\end{equation}
where $\bsx_1,\bsx_2,\ldots,\bsx_N$ are in $[0,1)^d$ and $a_1,a_2,\ldots,a_N$ are real weights that we call integration weights. If $a_1=a_2=\ldots =a_N=1/N$, then the linear algorithm \eqref{def:linAlg} is a so-called quasi-Monte Carlo algorithm, and we denote this by $A_{d,N}^{{\rm QMC}}$.

The worst-case error of an algorithm \eqref{def:linAlg} is defined by 
\begin{equation}\label{eq:wce}
e(F_{d,\bsgamma,q},A_{d,N})=\sup_{f \in F_{d,\bsgamma,q}\atop \|f\|_{d,\bsgamma,q}\le 1} \left|I_d(f)-A_{d,N}(f)\right|.
\end{equation}
For a quasi-Monte Carlo algorithm $A_{d,N}^{{\rm QMC}}$ it is known that 
$$
e(F_{d,\bsgamma,q},A_{d,N}^{{\rm QMC}})= L_{p,\bsgamma,N}(\overline{\cP}),
$$ 
where $L_{p,\bsgamma,N}(\overline{\cP})$ is the weighted $L_p$-discrepancy of the point set
\begin{equation}\label{def:oP}
\overline{\cP}=\{\boldsymbol{1} - \bsx_k \ : \ k\in \{1,2,\ldots,N\}\},
\end{equation}
where $\boldsymbol{1} - \bsx_k$ is defined as the component-wise difference of the vector containing only ones and $\bsx_k$, see, e.g., \cite[Section~1]{LP03}, \cite[Section~3]{SW98} and \cite[Section~9.5.1]{NW10}. We mention that we have chosen to anchor the norms \eqref{norm1} and \eqref{def:normd}, respectively, in $0$ whereas in \cite{LP03} and \cite{SW98} the anchor 1 is used. This change is not a big deal and results in the appearence of $\overline{\cP}$ in the above results rather than directly the point set $\cP$. 

For general linear algorithms  \eqref{def:linAlg} the worst-case error is the so-called generalized weighted $L_p$-discrepancy. Here, for points $\cP=\{\bsx_1,\bsx_2,\ldots,\bsx_N\}$ and corresponding coefficients $\cA=\{a_1,a_2,\ldots,a_N\}$ the discrepancy function is $$\overline{\Delta}_{\cP,\cA}(\bst)=\sum_{k=1}^N a_k {\bf 1}_{[\boldsymbol{0},\bst)}(\bsx_k) - t_1 t_2\cdots t_d$$ for $\bst=(t_1,t_2,\ldots,t_d)$ in $[0,1]^d$ and, for $p \in [1,\infty)$, the generalized $\bsgamma$-weighted $L_p$-discrepancy is
$$\overline{L}_{p,\bsgamma,N}(\cP,\cA):=\left(\sum_{\emptyset \not=\uu \subseteq [d]}\gamma_{\uu}^{p/2}\int_{[0,1]^{|\uu|}} |\overline{\Delta}_{\cP,\cA}((\bst_{\uu},1))|^p \rd \bst_{\uu}\right)^{1/p}.$$ If $a_1=a_2=\ldots=a_N=1/N$, then we are back to the classical definition of weighted $L_p$-discrepancy in Section~\ref{sec:intro}. 

Now a natural extension of the QMC-setting to arbitrary linear algorithms yields
$$
e(F_{d,\bsgamma,q},A_{d,N})= \overline{L}_{p,\bsgamma,N}(\overline{\cP},\cA),
$$ 
where $\overline{\cP}$ is like in \eqref{def:oP} and $\cA=\{a_1,a_2,\ldots,a_N\}$ consists of exactly the coefficients from the given linear algorithm \eqref{def:linAlg}. See \cite[Section~9.5.1]{NW10} for the case $p=2$.

From this point of view we now study the more general problem of numerical integration in the weighted space $F_{d,\bsgamma,q}$ rather than the $\bsgamma$-weighted $L_p$-discrepancy (which corresponds to quasi-Monte Carlo algorithms -- although with suitably ``reflected'' points). We consider linear algorithms where we restrict ourselves to non-negative weights $a_1,\ldots,a_N$ and hence QMC-algorithms 
are included in the present setting.

We define the $N$-th minimal worst-case error as $$e_{q,\bsgamma}(N,d):=\min_{A_{d,N}} |e(F_{d,\bsgamma,q},A_{d,N})|,$$ 
where the minimum is extended over all linear algorithms of the form \eqref{def:linAlg} based on $N$ function evaluations along points $\bsx_1,\bsx_2,\ldots,\bsx_N$ from $[0,1)^d$ and with non-negative weights $a_1,\ldots,a_N \ge 0$. Note that for all $d,N \in \N$ we have 
\begin{equation}\label{ine:errdisc}
e_{q,\bsgamma}(N,d) \le {\rm disc}_{p,\bsgamma}(N,d).
\end{equation} 

The initial error is $$e_{q,\bsgamma}(0,d)=\sup_{f \in F_{d,\bsgamma,q}\atop \|f\|_{d,\bsgamma,q}\le 1} \left|I_d(f)\right|.$$ 

We call a function $f \in F_{d,\bsgamma,q}$ a worst-case function, if $I_d(f/\|f\|_{d,\bsgamma,q})=e_{q,\bsgamma}(0,d)$.

\begin{lem}\label{le:interr}
Let $d \in \N$ and let $q \in (1,\infty]$ and $p \in [1,\infty)$ with $1/p+1/q=1$. Then we have $$e_{q,\bsgamma}(0,d)=\prod_{j=1}^d \left(1+\frac{\gamma_j}{p+1}\right)^{1/p}$$ and the worst-case function in $F_{d,\bsgamma,q}$ is given by $h_{d,\bsgamma}(\bsx)=h_{1,\gamma_1}(x_1) \cdots h_{1,\gamma_d}(x_d)$ for $\bsx=(x_1,\ldots,x_d) \in [0,1]^d$, where $h_{1,\gamma}(x)=1+\gamma \frac{1-(1-x)^p}{p}$. Furthermore, we have 
\begin{equation}\label{wcfct:intnor}
\int_0^1 h_{1,\gamma}(t)\rd t= 1+\frac{\gamma^{p/2}}{p+1}\quad \mbox{ and } \quad \|h_{1,\gamma}\|_{1,\gamma,q}= \left(1+\frac{\gamma^{p/2}}{p+1}\right)^{1/q}.
\end{equation}
\end{lem}

\begin{proof}
Since we are dealing with tensor products of one-dimensional spaces it suffices to prove the result for $d=1$. Let $\gamma>0$ be a generic weight. For $f \in F_{1,\gamma,q}$ we have $$\int_0^1 f(x) \rd x = \int_0^1 \left(f(0)+\int_0^x f'(t) \rd t\right) \rd x = f(0)+\int_0^1 f'(t) g(t) \rd t,$$ where $g(t)=1-t$.  Applying H\"older's inequality we obtain 
\begin{align*}
\left|\int_0^1 f(x) \rd x \right| \le & |f(0)| + \|f\|_{1,q} \|g\|_{L_p}\\
= & |f(0)| + \frac{1}{\gamma^{1/2}} \|f\|_{1,q} \gamma^{1/2}\|g\|_{L_p}\\
\le & \|f\|_{1,\gamma,q} \left(1+\gamma^{p/2} \|g\|_{L_p}^p\right)^{1/p}.
\end{align*}
Hence 
\begin{equation}\label{init:bd}
\left|\int_0^1 \frac{f(x)}{\|f\|_{1,\gamma,q}} \rd x \right|
\le \left(1+\gamma^{p/2} \|g\|_{L_p}^p\right)^{1/p} = \left(1+\frac{\gamma^{p/2}}{p+1}\right)^{1/p}.
\end{equation}

Now consider $h_{1,\gamma}(x)=1+\gamma \frac{1-(1-x)^p}{p}$. It is elementary to check that \eqref{wcfct:intnor} holds true. Hence, with $f(x)=h_{1,\gamma}(x)$ we obtain equality in \eqref{init:bd}. Thus, $$e_{q,\gamma}(0,1)=\left(1+\frac{\gamma^{p/2}}{p+1}\right)^{1/p}$$ and $h_{1,\gamma}$ is a worst-case function.
\end{proof}

Note that for all H\"older conjugates $q \in (1,\infty]$ and $p \in [1,\infty)$ and for all $d \in \N$ we have $$e_{q,\bsgamma}(0,d)={\rm disc}_{p,\bsgamma}(0,d).$$

Now we define the information complexity as the minimal number of function evaluations necessary in order to reduce the initial error by a factor of $\varepsilon$. 
For $d \in \N$ and $\varepsilon \in (0,1)$ put 
$$N^{{\rm int}}_{q,\bsgamma}(\varepsilon,d):= \min\{N \in \N \ : \ e_{q,\bsgamma}(N,d) \le \varepsilon\ e_{q,\bsgamma}(0,d)\}.$$
We stress that $N^{{\rm int}}_{q,\bsgamma}(\varepsilon,d)$
is a kind of restricted complexity since we only allow positive quadrature formulas.

From \eqref{ine:errdisc}, \eqref{initdisc} and Lemma~\ref{le:interr} it follows that for all  H\"older conjugates $q \in (1,\infty]$ and $p \in [1,\infty)$ and for all $d \in \N$ and $\varepsilon \in (0,1)^d$ we have  $$N^{{\rm int}}_{q,\bsgamma}(\varepsilon,d) \le N^{{\rm disc}}_{p,\bsgamma}(\varepsilon,d).$$

Hence, Theorem~\ref{thm1} follows from the following more general results.

\begin{thm}\label{thm2}
For every $q$ in $(1,\infty)$ put
\begin{equation}\label{def:taup}
\tau_p:=  \frac{2p-(p+1)(1+2^{p/(p+1)}-2^{1/(p+1)})}{4 p^2+6p + (p+1)(1+2^{p/(p+1)}-2^{1/(p+1)})},
\end{equation}
where $p$ is the H\"older conjugate of $q$. Then $\tau_p>0$ and for all $d \in \N$ and $\varepsilon \in (0,1/2)$ we have 
\begin{equation*}
N^{{\rm disc}}_{p,\bsgamma}(\varepsilon,d) \ge N_{q,\bsgamma}^{{\rm int}}(\varepsilon,d) \ge (1-2 \varepsilon) \prod_{j=1}^d \left(1+\tau_p \gamma_j^{p/2}\right)^{1/q}.
\end{equation*}
\end{thm}

The proof of Theorem~\ref{thm2} will be given in Section~\ref{sec:proofs}. It will follow from a general lower bound on the information complexity of  the integration problem in (not necessarily homogeneous) linear tensor products of spaces of univariate functions for positive linear rules. This general result will be explained in Section~\ref{sec:genres}.

The notions of weak, polynomial and strong polynomial tractability for integration in $F_{d,\bsgamma,q}$
are defined in the same way as for the weighted $L_p$-discrepancy 
in Section~\ref{sec:intro} with the inverse discrepancy 
$N_{p,\bsgamma}^{{\rm disc}}(\varepsilon,d)$ replaced by the 
information complexity $N^{{\rm int}}_{q,\bsgamma}(\varepsilon,d)$. 
We obtain the following tractability conditions for the integration problem 
with positive quadrature formulas from Theorem~\ref{thm2}.
The proof of Corollary~\ref{cor2} will be given in Section~\ref{sec:proofs}.

\begin{cor}\label{cor2}
Let $q \in (1,\infty)$ and let $p$ denote the H\"older conjugate of $q$. 
For coordinate weights  
$\bsgamma=(\gamma_j)_{j \ge 1}$ in $(0,1]$ we have: 
\begin{itemize} 
\item
A necessary condition for any notion of tractability for integration in 
$F_{d,\bsgamma,q}$ for positive (integration) weights is $\lim_{j \rightarrow \infty}\gamma_j=0$. 

\item
A necessary condition for
strong polynomial tractablity is 
 \begin{equation}\label{cond:SPTint}
 \sum_{j=1}^{\infty}\gamma_j^{p/2}<\infty.
 \end{equation}
If $q$ of the form $q=2\ell/(2 \ell -1)$ with $\ell \in \N$, then this condition 
is also sufficient. 
  
 \item 
 For every $q \in (1,\infty)$ integration in $F_{d,\bsgamma,q}$ for positive 
 rules is polynomially tractable, if and only if 
  \begin{equation}\label{cond:PTint}
 \limsup_{d \rightarrow \infty}\frac{1}{\log d}\sum_{j=1}^d \gamma_j^{p/2} < \infty
 \end{equation}
holds. This condition is sufficient for polynomial tractability also for $q=\infty$.

 \item
For every $q \in (1,\infty)$ integration in $F_{d,\bsgamma,q}$ for positive rules is 
weakly tractable, if and only if 
 \begin{equation}\label{cond:WTint}
 \lim_{d \rightarrow \infty}\frac{1}{d}\sum_{j=1}^d \gamma_j^{p/2}=0
 \end{equation}
holds. This condition is sufficient for weak tractability also for  $q=\infty$.

\end{itemize}
\end{cor}

Note that the $q$'s considered in the second  item of Corollary~\ref{cor2} 
are the H\"older conjugates of even integers $p$. All these $q$'s belong to $(1,2]$.

\section{A general lower bound}\label{sec:genres}

In this section we provide a general lower bound on the information complexity of the integration problem in (not necessarily homogeneous) tensor products of spaces of univariate functions. This general result will be used to prove Theorem~\ref{thm2} (and thus Theorem~\ref{thm1}). In Section~\ref{sec:appl2} we will provide another application of this powerful result.

For $j \in \N$ let $(F^{(j)},\|\cdot\|^{(j)})$ be normed spaces of univariate integrable functions over the Borel measurable domain $D^{(j)} \subseteq \R$ and let, for $d \in \N$, $(F_d,\|\cdot\|_d)$ be the $d$-fold tensor product space $$F_d:=F^{(1)}\otimes \ldots \otimes F^{(d)},$$ which consists of functions over the domain $D_d:=D^{(1)}\times \ldots \times D^{(d)}$, equipped with a crossnorm $\|\cdot\|_d$, i.e., for elementary tensors of the form $f(x_1,\ldots,x_d)=f_1(x_1)\cdots f_d(x_d)$ with $f_j:D^{(j)} \rightarrow \R$ for $j \in [d]$, we have $\|f\|_d=\|f_1\|^{(1)} \cdots \|f_d\|^{(d)}$. 

Consider multivariate integration $$I_d(f):=\int_{D_d} f(\bsx) \rd \bsx \quad \mbox{for $f \in F_d$}.$$ The so-called initial error of this integration problem is 
\begin{equation}\label{def:int:err}
e(0,d)=\|I_d\|=\sup_{f \in F_d\atop \|f\|_d \le 1} \left|I_d(f)\right|.
\end{equation}
We assume that for every $j \in [d]$ there exists a so-called worst-case function $h^{(j)}$ in $F^{(j)}$, which is a function that satisfies $I_1(h^{(j)})= \|h^{(j)}\|^{(j)} e_j(0,1)$, 
where $e_j(0,1)$ is the initial error for univariate integration in $F^{(j)}$.
Furthermore, we also assume that the tensor product $h_d(\bsx):=h^{(1)}(x_1) \cdots h^{(d)}(x_d)$ is a worst-case function for $d>1$ and hence $e(0,d)=\prod_{j=1}^d e_j(0,1)$ is the initial error for integration in $F_d$.

We approximate integrals by linear algorithms of the form 
\begin{equation}\label{def:linAlg3}
A_{d,N}(f)=\sum_{k=1}^N a_k f(\bsx_k),
\end{equation}
where $\bsx_1,\bsx_2,\ldots,\bsx_N$ are in $D_d$ and $a_1,a_2,\ldots,a_N$ are non-negative integration weights\footnote{We confess that non-negativity of integration weights constitutes a certain restriction which is in general not legitimated by theory, but which is often favored by practitioners
because of the stability of positive quadrature formulas.}. 
The worst-case error of an algorithm is defined by
\begin{equation*}
e(F_d,A_{d,N})=\sup_{f \in F_d\atop \|f\|_d\le 1} \left|I_d(f)-A_{d,N}(f)\right|.
\end{equation*}
We define the $N$-th minimal worst-case error for positive linear rules as 
$$e(N,d):=\min_{A_{d,N}} e(F_d,A_{d,N}),$$ 
where the minimum is extended over all linear algorithms of the form \eqref{def:linAlg3} based on $N$ function evaluations along points $\bsx_1,\bsx_2,\ldots,\bsx_N$ from $D_d$ and with non-negative integration weights $a_1,\ldots,a_N$. 

For $\varepsilon \in (0,1)$ and $d \in \mathbb{N}$ define $$N(\varepsilon,d)=\min\{N \in \mathbb{N} \ : \ e(N,d) \le \varepsilon \, e(0,d)\}.$$ We stress that $N(\varepsilon,d)$ is a kind of restricted complexity since we only allow positive quadrature formulas.

The following result can be seen as the non-homogeneous extension of \cite[Theorem~4]{NP24}.

\begin{thm}\label{thm:gen}
For every $j \in [d]$: assume that there exists a worst-case function $h^{(j)}$ in $F^{(j)}$ and assume that $h_d(\bsx)=h^{(1)}(x_1)\cdots h^{(d)}(x_d)$ for $\bsx=(x_1,\ldots,x_d)\in D_d$ is a worst-case function in $F_d$.

Assume further that for every $y \in D^{(j)}$ there exists a function $s_y^{(j)} \in F^{(j)}$ such that $s_y^{(j)} \ge 0$ and $s_y^{(j)}(y)=h^{(j)}(y)$. Put
\begin{equation}\label{est:alphabeta}
\alpha_j:= \max_{y \in D^{(j)}} \|s_y^{(j)}\|^{(j)} \quad \mbox{and}\quad \beta_j:= \max_{y \in D^{(j)}} \int_{D^{(j)}} s_y^{(j)}(x) \rd x,
\end{equation} 
and 
\begin{equation}\label{def:C1new}
\bsalpha:=\prod_{j=1}^d \alpha_j, \quad \bsbeta :=\prod_{j=1}^d \beta_j, \quad \mbox{ and } \quad \widetilde{C}_d:=\min\left(\frac{\|h_d\|_d}{\bsalpha},\frac{I_d(h_d)}{\bsbeta}\right).
\end{equation}
Then $$N(\varepsilon,d) \ge \widetilde{C}_d\, (1-2 \varepsilon) \quad \mbox{for all $\varepsilon \in \left(0,\frac{1}{2}\right)$ and $d \in \N$.}$$ 
\end{thm}

\begin{rem}\rm
When applying Theorem~\ref{thm:gen}, strive for functions $s_y^{(j)}$ such that $\alpha_j < \|h^{(j)}\|^{(j)}$ and $\beta_j < \int_{D^{(j)}}  h^{(j)}(x)\rd x$, which yields that 
$$\min\left(\frac{\|h^{(j)}\|^{(j)}}{\alpha_j},\frac{I_1(h^{(j)})}{\beta_j}\right)>1 \quad 
\mbox{at least for many $j \in [d]$.}$$ 
The smaller $\alpha_j$ and $\beta_j$ are, the larger is the lower bound $\widetilde{C}_d$.
\end{rem}

\begin{proof}[Proof of Theorem~\ref{thm:gen}]
Consider an algorithm $A_{d,N}$ of the form \eqref{def:linAlg3} based on nodes $\bsx_1,\ldots,\bsx_N$ in $D_d$ and with non-negative integration weights $a_1,\ldots,a_N\ge 0$.  Let $x_{i,j}$ be the $j$-th coordinate of the point $\bsx_i$, $i \in \{1,\ldots,N\}$ and $j \in [d]$. For $i \in \{1,\ldots,N\}$ we define functions 
$$
P_i(\bsx) = s_{x_{i,1}}^{(1)}(x_1) s_{x_{i,2}}^{(2)}(x_2) \cdots s_{x_{i,d}}^{(d)}(x_d),\quad \mbox{for $\bsx=(x_1,\ldots,x_d)\in D_d$.}  
$$

In order to estimate the error of $A_{d,N}$ we consider the two functions $h_d$ and $f^*:= \sum_{i=1}^N P_i$. Since $A_{d,N}$ is a positive rule we have $A_{d,N}(f^*) \ge A_{d,N}(h_d)$. Then we have the error estimate 
\begin{equation}\label{errest1linneu}
e(F_d,A_{d,N})  \ge \frac{ \left(I_d(h_d) - I_d(f^*)\right)_+}{2 \max ( \Vert h_d \Vert_d, \Vert f^* \Vert_d)},
\end{equation}
which is trivially true if $I_d(h_d) \le  I_d(f^*)$ and which is easily shown if $I_d(h_d) >  I_d(f^*)$, because then 
\begin{align*}
\left(I_d(h_d) - I_d(f^*)\right)_+  \le & I_d(h_d) - A_{d,N}(h_d)+A_{d,N}(f^*)-I_d(f^*)\\
 \le & 2 \max(\|h_d\|_d,\|f^*\|_d) \, e(F_d,A_{d,N}).
\end{align*}

According to \eqref{est:alphabeta} we have $\Vert f^* \Vert_d \le N \bsalpha$ and $I_d(f^*) \le N  \bsbeta$. Inserting into \eqref{errest1linneu} yields 
\begin{equation}\label{lberr17}
e(N,d) \ge e(F_d,A_{d,N}) \ge \frac{(I_d(h_d) - N \bsbeta)_+} {2 \max(\|h_d\|_d,N \bsalpha)} = e(0,d) \frac{(1 - N \bsbeta/I_d(h_d))_+} {2 \max(1,N \bsalpha/\|h_d\|_d)}.
\end{equation}

Now let $\varepsilon \in (0,1/2)$ and assume that $e(N,d) \le \varepsilon \, e(0,d)$. This implies 
\begin{equation}\label{lberr17a}
2\, \varepsilon \max\left(1,N \frac{\bsalpha}{\|h_d\|_d}\right) \ge  \left(1 - N \frac{\bsbeta}{I_d(h_d)}\right)_+.
\end{equation}
If $N \le \widetilde{C}_d$, where $\widetilde{C}_d$ is from \eqref{def:C1new}, then we obtain $$N \ge \frac{I_d(h_d)}{\bsbeta} \, (1-2 \varepsilon) \ge \widetilde{C}_d \, (1-2 \varepsilon).$$ If $N > \widetilde{C}_d$, then we trivially have $N \ge \widetilde{C}_d \, (1-2 \varepsilon)$. 
This yields $$N(\varepsilon,d)\ge \widetilde{C}_d \, (1-2 \varepsilon),$$ and we are done.
\end{proof} 

\section{The proofs of Theorem~\ref{thm2} and Corollaries~\ref{cor2} and \ref{cor1}}\label{sec:proofs}

In this section we present the proofs of Theorem~\ref{thm2} (and hence of Theorem~\ref{thm1}) and subsequently that of Corollary~\ref{cor2} and Corollary~\ref{cor1}.

\begin{proof}[Proof of Theorem~\ref{thm2}] We apply Theorem~\ref{thm:gen} to the specific setting in Section~\ref{sec:int}.  Let $(F^{(j)},\|\cdot\|^{(j)})=(F_{1,\gamma_j,q},\|\cdot \|_{1,\gamma_j,q})$ and $(F_d,\|\cdot\|_d)=(F_{d,\bsgamma,q},\|\cdot \|_{d,\bsgamma,q})$. According to Lemma~\ref{le:interr} a worst-case function in $F_{1,\gamma_j,q}$ is $$h_{1,\gamma_j}(x)=1+\gamma_j \frac{1-(1-x)^p}{p}.$$

Let $a \in (0,1)$ that will be specified later and define
$$
s_y^{(j)}(x)=1+\frac{\gamma_j^{p/2}}{p} \times \left\{ 
\begin{array}{ll}
h_{1,1}(x) + h_{1,2,(0)}(x) & \mbox{if $y \le a$,}\\
h_{1,1}(x) + h_{1,2,(1)}(x) & \mbox{if $y \ge a$,}
\end{array}
\right.
$$
where 
\begin{align*}
h_{1,1}(x)  = & \frac{1-(1-a)^p}{a} \min(x,a),\\
h_{1,2,(0)}(x)  = & \mathbf{1}_{[0,a]}(x)\left((1-(1-x)^p)-\frac{x}{a}(1-(1-a)^p) \right),\\
h_{1,2,(1)}(x)  = & \mathbf{1}_{[a,1]}(x)\left((1-(1-x)^p)-(1-(1-a)^p) \right).
\end{align*}
Observe that $h_{1,1}(x)+h_{1,2,(0)}(x)+h_{1,2,(1)}(x)=1-(1-x)^p$ and that $h_{1,2,(0)}$ and $h_{1,2,(1)}$ have a disjoint support, namely $[0,a)$ and $(a,1]$, respectively.

Obviously, $s_y^{(j)}$ is non-negative, belongs to $F_{1,\gamma_j,q}$ and satisfies $s_y^{(j)}(y)=h^{(j)}(y)$ for all $y \in [0,1]$. Thus we can apply Theorem~\ref{thm:gen}, which yields, for $\varepsilon \in (0,1/2)$, $$N_{q,\bsgamma}^{{\rm int}}(N,d) \ge \widetilde{C}_d \, (1-2 \varepsilon),$$ with $\widetilde{C}_d$ like in \eqref{def:C1new}. It remains to estimate $\widetilde{C}_d$. 

To begin with, it is elementary to compute 
$$\left\Vert 1+\frac{\gamma_j^{p/2}}{p}(h_{1,1} + h_{1,2,(0)}) \right\Vert_{1,\gamma_j,q}^q= 1+\gamma_j^{p/2} \frac{1-(1-a)^{p+1}}{p+1}$$ and $$\left\Vert 1+\frac{\gamma_j^{p/2}}{p}(h_{1,1} + h_{1,2,(0)}) \right\Vert_{1,\gamma_j,q}^q= 1+\gamma_j^{p/2} \frac{(1-a)^{p+1}}{p+1}.$$ Thus, both quantities coincide if we choose 
\begin{equation}\label{def:a}
a=1-\frac{1}{2^{1/(p+1)}}
\end{equation}
and then $$\alpha_j=\left(1+\frac{\gamma_j^{p/2}}{2(p+1)}\right)^{1/q} \quad \mbox{and}\quad \bsalpha= \prod_{j=1}^d \left(1+\frac{\gamma_j^{p/2}}{2(p+1)}\right)^{1/q}.$$ Hence
$$\frac{\|h_{d,\bsgamma}\|_{d,\bsgamma,q}}{\bsalpha}=\prod_{j=1}^d \left(\frac{1+\frac{\gamma_j^{p/2}}{p+1}}{1+\frac{\gamma_j^{p/2}}{2(p+1)}} \right)^{1/q} \ge \prod_{j=1}^d \left(1+\frac{\gamma_j^{p/2}}{2 p+3} \right)^{1/q},$$ where in the last step we used that $\gamma_j \le 1$.

Furthermore, 
$$
\beta_j = 1+\frac{\gamma_j^{p/2}}{p} \max \left( \int_0^1 h_{1,1}(x) + h_{1,2,(0)}(x)\rd x , \int_0^1 h_{1,1}(x) + h_{1,2,(1)}(x)\rd x \right). 
$$
For $a$ like in \eqref{def:a} we have 
$$\int_0^1 h_{1,1}(x)+h_{1,2,(0)}(x) \rd x = \frac{1}{2} \frac{p}{p+1}$$ and $$\int_0^1 h_{1,1}(x)+h_{1,2,(1)}(x) \rd x = \frac{1}{2} \frac{p}{p+1} + \frac{1+2^{p/(p+1)}-2^{1/(p+1)}}{4}.$$ Hence $$\beta_j=1+\frac{\gamma_j^{1/p}}{2(p+1)} \left(1 + \frac{p+1}{2p} (1+2^{p/(p+1)}-2^{1/(p+1)}) \right).$$ Therefore we obtain 
\begin{align*}
\frac{I_d(h_{d,\bsgamma})}{\bsbeta}&=\prod_{j=1}^d \frac{1+\frac{\gamma_j^{p/2}}{p+1}}{1+\frac{\gamma_j^{1/p}}{2(p+1)} \left(1 + \frac{p+1}{2p} (1+2^{p/(p+1)}-2^{1/(p+1)}) \right)}\\
& \ge \prod_{j=1}^d \left(1+\gamma_j^{p/2} \tau_p\right),
\end{align*} 
where we again used that $\gamma_j \le 1$ and where $\tau_p$ is defined in \eqref{def:taup}. Since $\frac{p+1}{2p} (1+2^{p/(p+1)}-2^{1/(p+1)})<1$ it is clear that $\tau_p>0$ (note that $p \in (1,\infty)$).

In any case, we have
\begin{align*}
\widetilde{C}_q & \ge \prod_{j=1}^d \left(1+\tau_p \gamma_j^{p/2}\right)^{1/q}
\end{align*}
and this finishes the proof of Theorem~\ref{thm2}.
\end{proof}

\begin{proof}[Proof of Corollary~\ref{cor2}]
We use standard arguments. Assume first that there exists a real $\Gamma>0$ such that $\gamma_j\ge \Gamma$ for all $j \in \N$. Then it follows from Theorem~\ref{thm2} that $$N_{q,\bsgamma}^{{\rm int}}(\varepsilon,d) \ge (1-2\varepsilon) \left(1+\tau_p \Gamma^{p/2}\right)^{d/q}.$$ This means that the integration problem suffers from the curse of dimensionality. Thus, $\lim_{j \rightarrow \infty}\gamma_j=0$ is a necessary condition for tractability.

Now suppose that $\sum_{j=1}^{\infty} \gamma_j^{p/2} =\infty$. Since $\prod_{j=1}^d (1+\tau_p \gamma_j^{p/2}) \ge 1+\tau_p \sum_{j=1}^d \gamma_j^{p/2}$ it follows that for $\varepsilon \in (0,1/2)$ we have $$\lim_{d \rightarrow \infty} N_{q,\bsgamma}^{{\rm int}}(\varepsilon,d) =\infty$$ and hence we cannot have strong polynomial tractability. Thus, \eqref{cond:SPTint} is a necessary condition for strong polynomial tractability.

Next suppose that $\limsup_{d \rightarrow \infty} \frac{1}{\log d} \sum_{j=1}^d \gamma_j^{p/2} =\infty$. Note that $0< \gamma_j^{p/2} \tau_p \le \tau_p<1$ and $\lim_{x \rightarrow \infty} \tfrac{1}{x}\log(1+x) =1$. Hence, there exists a positive real $\ell$ such that $\log(1+\gamma_j^{p/2} \tau_p) > \ell \gamma_j^{p/2} \tau_p$ for all $j \in \N$. This gives $$\log \prod_{j=1}^d\left(1+\gamma_j^{p/2} \tau_p\right)^{1/q} \ge \frac{\ell \tau_p}{q} \sum_{j=1}^d \gamma_j^{p/2}$$ and hence $$\prod_{j=1}^d\left(1+\gamma_j^{p/2} \tau_p\right)^{1/q} \ge \exp\left(\frac{\ell \tau_p}{q} \sum_{j=1}^d \gamma_j^{p/2}\right)=d^{\frac{\ell \tau_p}{q} \frac{1}{\log d}\sum_{j=1}^d \gamma_j^{p/2}}. $$ Hence, $N_{q,\bsgamma}^{{\rm int}}(\varepsilon,d)$ goes to infinity faster than any power of $d$ and we cannot have polynomial tractability. This means that \eqref{cond:PTint} is a necessary condition for polynomial tractability.

Finally, similarly to the above it follows from Theorem~\ref{thm2} that $$\log N_{q,\bsgamma}^{{\rm int}}(\varepsilon,d) \ge \frac{\tau_p \ell}{q} \sum_{j=1}^d \gamma_j^{p/2}$$ and hence $$\frac{\log N_{q,\bsgamma}^{{\rm int}}(\varepsilon,d) }{d+\varepsilon^{-1}} \ge \frac{\tau_p \ell}{q} \frac{d}{d+\varepsilon^{-1}} \frac{1}{d} \sum_{j=1}^d \gamma_j^{p/2}.$$ Thus, weak tractability implies \eqref{cond:WTint}.

The sufficient conditions follow directly from 
Corollary~\ref{cor1}, which will be proven immediately afterwards.
\end{proof}

\begin{proof}[Proof of Corollary~\ref{cor1}]
The sufficiency of \eqref{cond:SPT} for strong polynomial tractability for {\it even} integers $p$ was shown in \cite{LP03}. It remains to prove the sufficient conditions for polynomial- and weak tractablilty, respectively, for general $p \in [1,\infty)$. To this end, we first show a general upper bound on the inverse of the $N$-th minimal weighted $L_p$-discrepancy.

Let $p \in [1,\infty)$. According to the definition in \eqref{def:wLpdisc} we have
\begin{equation*}
L_{p,\bsgamma,N}(\cP)\le \left(\sum_{\emptyset \not=\uu \subseteq [d]}\gamma_{\uu}^{p/2}\left( D_N^{\ast}(\cP_{\uu}) \right)^p \right)^{1/p},
\end{equation*}
where $\cP_{\uu}$ is the set of projections of the points from $\cP$ to the coordinate directions which belong to $\uu$, and $D_N^{\ast}$ is the usual star-discrepancy of a point set, i.e., $D_N^{\ast}(\cP)=\|\Delta_{\cP}\|_{L_{\infty}}$. It is shown in \cite[Proof of Theorem~1]{HPS08} that there exists an absolute constant $C>0$ such that for every $d, N \in \N$ there exists an $N$-point set $\cP$ in $[0,1)^d$ such that $$D_N^{\ast}(\cP_{\uu}) \le C \max(1,\sqrt{\log d}) \, \sqrt{\frac{|\uu|}{N}} \quad \mbox{for all non-empty $\uu \subseteq [d]$.}$$ Therefore, and since $|\uu| \le 2^{|\uu|}$ for $\uu \not=\emptyset$, we obtain
\begin{align*}
{\rm disc}_{p,\bsgamma}(N,d) & \le C\, \frac{\max(1,\sqrt{\log d})}{\sqrt{N}}\left(\sum_{\emptyset \not=\uu \subseteq [d]}\gamma_{\uu}^{p/2} 2^{|\uu| p/2} \right)^{1/p}\\
& \le C\, \frac{\max(1,\sqrt{\log d})}{\sqrt{N}} \prod_{j=1}^d \left(1+ \gamma_j^{p/2} 2^{p/2}\right)^{1/p}.
\end{align*}
This estimate together with the observation that ${\rm disc}_{p,\bsgamma}(0,d) >1$ implies that 
\begin{equation}\label{ubd:infdisc1}
N_{p,\bsgamma}^{{\rm disc}}(\varepsilon,d) \le \varepsilon^{-2} \, C^2 \, \max(1,\log d) \, \prod_{j=1}^d \left(1+ \gamma_j^{p/2} 2^{p/2}\right)^{2/p}.
\end{equation}
This is the promised estimate of the inverse of the $N$-th minimal weighted $L_p$-discrepancy. We employ \eqref{ubd:infdisc1} in order to prove sufficiency of the proposed conditions on the weights.

First, assume that \eqref{cond:PT} holds. Then for every $\delta >0$ there exists a $d_{\delta}\in \N$ such that for all $d\ge d_{\delta}$ we have $$\frac{1}{\log d} \sum_{j=1}^d \gamma_j^{p/2} < A+\delta.$$ Then, for $d \ge d_{\delta}$, we get
\begin{eqnarray*}
\prod_{j=1}^d \left(1+ \gamma_j^{p/2} 2^{p/2}\right) \le \exp\left(2^{p/2} \sum_{j=1}^d \gamma_j^{p/2} \right) = d^{2^{p/2} (\sum_{j=1}^d \gamma_j^{p/2})/\log d} \le d^{2^{p/2} (A+\delta)}. 
\end{eqnarray*} 
Therefore, there exists a positive real $c_{\delta}$ such that we have $$\prod_{j=1}^d \left(1+ \gamma_j^{p/2} 2^{p/2}\right) \le c_{\delta}\, d^{2^{p/2} (A+\delta)} \quad \mbox{for all $d \in \N$.}$$ Inserting in \eqref{ubd:infdisc1} we obtain
$$N_{p,\bsgamma}^{{\rm disc}}(\varepsilon,d) \le  C^2 \,  c_{\delta}^{2/p}\, \max(1,\log d) \, d^{2^{1+p/2} (A+\delta)/p}\, \varepsilon^{-2}$$ and hence we have polynomial tractability.

Finally, assume that \eqref{cond:WT} holds. For $d \ge 2$ we have
\begin{equation*}
\frac{N_{p,\bsgamma}^{{\rm disc}}(\varepsilon,d)}{d+\varepsilon^{-1}} \le \frac{2 \log \varepsilon^{-1} + 2 \log C + \log\log d + \frac{2}{p} 2^{p/2} \sum_{j=1}^d \gamma_j^{p/2}}{d+\varepsilon^{-1}},
\end{equation*}
where the right hand side tends to 0 whenever $d+\varepsilon^{-1} \rightarrow \infty$, according to \eqref{cond:WT}. Hence we have weak tractability.
\end{proof}

Unfortunately, because of the annoying $\log d$-term, we cannot use \eqref{ubd:infdisc1} in order to prove sufficient conditions for strong polynomial tractability. Using \eqref{ubd:infdisc1}, then, for $p \in [1,\infty)$, condition \eqref{cond:SPT} only implies polynomial tractability, however, with $\varepsilon$-exponent zero\footnote{More precisely, we have generalized $(T,\Omega)$-tractability with $T(\varepsilon^{-1},d)=\varepsilon^{-1} (1+\log d)$ and $\Omega=[1,\infty)\times \N$. See \cite[Section~4.4.3]{NW08} for the notion of generalized tractability.}.

\section{A further application of Theorem~\ref{thm:gen}: integration of polynomials}\label{sec:appl2}

We give another application in order to demonstrate the power of Theorem~\ref{thm:gen}. Here we study integration of $C^{\infty}$-functions, in particular, we restrict to real polynomials of degree at most two over the unit interval. 

Let $\bsgamma=(\gamma_j)_{j \ge 1}$ be a sequence of positive coordinate weights and let $q \in (1,\infty)$. Sticking to the notation in Section~\ref{sec:genres}, for $j \in \N$ consider $F^{(j)}=P_2$ as the space of all real polynomials of degree at most 2 as functions over $[0,1]$ with finite ``weighted'' $q$-norm $$\|f\|^{(j)}=\|f\|_{q,\gamma_j}:=\left(\|f\|_{L_q}^q+\frac{1}{\gamma_j}\left(\|f'\|_{L_q}^q+\|f''\|_{L_q}^q\right)\right)^{1/q}.$$ Consider integration $I_d(f)=\int_{[0,1]^d }f(\bsx)\rd \bsx$ for $f$ in the tensor product space $F_d$. This problem is studied in \cite[Section~10.5.4]{NW10} (for the Hilbert-space case $q=2$) and in \cite{NP24b} (for general $q \in (1,\infty)$), but in both instances only for the unweighted setting (i.e., $\gamma_j \equiv 1$). It is shown that then 
the problem suffers from the curse of dimensionality. Hence it is no real restriction when we restrict ourselves to coordinate weights $\gamma_j \in (0,1]$ in the following.

The initial error for the present problem is $e(0,d)=1$ and the (univariate) worst-case function is $h^{(j)}(x)=1$ for $x \in [0,1]$ for every $j$. Obviously $\|h^{(j)}\|_{q,\gamma_j}=1$ and $\int_0^1 h^{(j)}(x)\rd x =1$. 

We can exploit Theorem~\ref{thm:gen} in order to show the following result.

\begin{thm}\label{thm7}
Let $\bsgamma=(\gamma_j)_{j \ge 1}$ be a sequence of positive coordinate weights at most one and let $q \in (1,\infty)$. Consider integration over the unit-cube in the $d$-fold tensor product of spaces $(P_2,\|\cdot\|_{q,\gamma_j})$.  Put 
\begin{equation}\label{def:rho}
\rho_q:= \frac{u_q}{12} - (2 u_q)^q \frac{q+2}{q+1},
\end{equation}
with some positive number $u_q$ that is strictly less than  $\left(\frac{1}{12 \cdot 2^q} \frac{q+1}{q+2}\right)^{1/(q-1)}$.

Then we have $\rho_q>0$ and $$N(\varepsilon,d) \ge (1-2\varepsilon)\, \prod_{j=1}^d \left(1+\gamma_j^{1/(q-1)} \rho_q\right)^{1/q}  \quad \mbox{for all $\varepsilon \in \left(0,\frac{1}{2}\right)$ and $d \in \N$}.$$ In particular, considering integration with positive rules, a necessary condition for 
\begin{itemize}
\item strong polynomial tractability is $$\sum_{j=1}^{\infty} \gamma_j^{1/(q-1)}< \infty;$$
\item polynomial tractability is $$\limsup_{d \rightarrow \infty}\frac{1}{\log d}\sum_{j=1}^d \gamma_j^{1/(q-1)} < \infty;$$
 \item weak tractability is $$\lim_{d \rightarrow \infty}\frac{1}{d}\sum_{j=1}^d \gamma_j^{1/(q-1)}=0.$$
\end{itemize}
\end{thm}

\begin{proof}
In order to apply Theorem~\ref{thm:gen} we define, for $j \in \N$, $s_y^{(j)}(x):=1-c_j(x-y)^2$, $x \in [0,1]$, with some quantity $c_j \in (0,1/2)$ that will be specified later. Then $s_y^{(j)}$ is non-negative, belongs to $F^{(j)}$ and, obviously, $s_y^{(j)}(y)=1=h^{(j)}(y)$ for every $y \in [0,1]$. Thus, Theorem~\ref{thm:gen} implies $$N(\varepsilon,d) \ge \widetilde{C}_d \, (1-2\varepsilon) \quad \mbox{for all $\varepsilon \in \left(0,\frac{1}{2}\right)$ and $d \in \N$},$$ where $\widetilde{C}_d=\min\left(\bsalpha^{-1},\bsbeta^{-1}\right)$, with $\bsalpha$ and $\bsbeta$ as defined in \eqref{def:C1new}. It remains to estimate these two quantities.

To begin with, it  is easy to show that $$I_1(s_y^{(j)})=1-c_j\left(y^2-y+\frac{1}{3}\right)$$ and hence $\beta_j=1-\frac{c_j}{12}$ and $\bsbeta=\prod_{j=1}^d \left(1-\frac{c_j}{12}\right)$.

In order to estimate $\bsalpha$ we need to estimate $\|s_y^{(j)}\|_{q,\gamma_j}$. Since $s_y^{(j)}(x) \in [0,1]$ and $q > 1$ we have $$\|s_y^{(j)}\|_{L_q}^q =\int_0^1 |s_y^{(j)}(x)|^q \rd x \le I_1(s_y^{(j)}) \le \beta_j = 1-\frac{c_j}{12}.$$ Furthermore, since $(s_y^{(j)}(x))'=2 c_j (y-x)$ and $(s_y^{(j)}(x))''=-2c_j$ we have $$\|(s_y^{(j)})'\|_{L_q}^q=\int_0^1|2 c_j(y-x)|^q \rd x=(2 c_j)^q \frac{y^{q+1}+(1-y)^{q+1}}{q+1} \le \frac{(2 c_j)^q}{q+1}$$ and $$\|(s_y^{(j)})''\|_{L_q}^q=\int_0^1|-2 c_j|^q \rd x = (2c_j)^q.$$ Hence $$\|s_y^{(j)}\|_{q,\gamma_j}^q \le  1-\frac{c_j}{12} + \frac{(2 c_j)^q}{\gamma_j} \frac{q+2}{q+1}.$$ For $q \in (1,\infty)$  a necessary condition such that the latter is strictly less than 1, i.e., strictly less than $\|h^{(j)}\|_{q,\gamma_j}$, is $$c_j< \left(\frac{1}{12 \cdot 2^q} \frac{q+1}{q+2}\right)^{1/(q-1)} \gamma_j^{1/(q-1)}.$$ Now we fix the initially unspecified quantity $c_j$ by setting $c_j:= u_q \gamma_j^{1/(q-1)}$ for $j \in \N$, where $u_q$ is a positive real that is strictly less than $\left(\frac{1}{12 \cdot 2^q} \frac{q+1}{q+2}\right)^{1/(q-1)}$. Note that $c_j \in (0,1/2)$ for $q \in (1,\infty)$, where we use that $\gamma_j \le 1$. Then $$\alpha_j \le \left(1-\frac{u_q \gamma_j^{1/(q-1)}}{12} + \frac{(2 u_q \gamma_j^{1/(q-1)})^q}{\gamma_j} \frac{q+2}{q+1} \right)^{1/q}.$$ Note that also  $$\beta_j=1-\frac{u_q \gamma_j^{1/(q-1)}}{12} \le \left(1-\frac{u_q \gamma_j^{1/(q-1)}}{12} + \frac{(2 u_q \gamma_j^{1/(q-1)})^q}{\gamma_j}  \frac{q+2}{q+1} \right)^{1/q}.$$ Thus,
\begin{eqnarray}\label{lbd:polyN}
\widetilde{C}_q & = & \min\left(\frac{1}{\prod_{j=1}^d \alpha_j} , \frac{1}{\prod_{j=1}^d \beta_j}\right)\nonumber\\
& \ge & \prod_{j=1}^d \left(1-\frac{u_q \gamma_j^{1/(q-1)}}{12} + \frac{(2 u_q \gamma_j^{1/(q-1)})^q}{\gamma_j}  \frac{q+2}{q+1} \right)^{-1/q}\nonumber \\
& = & \prod_{j=1}^d \left(1-\gamma_j^{1/(q-1)} \rho_q\right)^{-1/q}\nonumber \\
& \ge & \prod_{j=1}^d \left(1+\gamma_j^{1/(q-1)} \rho_q\right)^{1/q},
\end{eqnarray}
where $\rho_q$ is defined in \eqref{def:rho}. Note that $\rho_q>0$, because $u_q$ is strictly less than $\left(\frac{1}{12 \cdot 2^q} \frac{q+1}{q+2}\right)^{1/(q-1)}$. 

The necessary conditions for the three notions of tractability follow from \eqref{lbd:polyN} in the same way like in the proof of Corollary~\ref{cor2}.
\end{proof}

We leave it as an open problem to show that the necessary conditions of Theorem~\ref{thm7} are also sufficient.

\paragraph{Funding.} No funds, grants, or other support was received.

\paragraph{Data availability.} Data sharing is not applicable to this article as no datasets were generated or analysed during the current study.

\vspace{0.5cm}
\noindent{\bf Author's Address:}\\

\noindent Erich Novak, Mathematisches Institut, FSU Jena, Ernst-Abbe-Platz 2, 07740 Jena, Germany. Email: erich.novak@uni-jena.de\\

\noindent Friedrich Pillichshammer, Institut f\"{u}r Finanzmathematik und Angewandte Zahlentheorie, JKU Linz, Altenbergerstra{\ss}e 69, A-4040 Linz, Austria. Email: friedrich.pillichshammer@jku.at

\end{document}